\documentclass[oneside,english,reqno, 12pt]{amsart}

\usepackage{amssymb,amsfonts}
\usepackage{hyperref} 
\usepackage{enumerate}
\usepackage{mathrsfs}
\usepackage[a4paper, left=2.5cm, right=2.5cm, top=2.5cm, bottom=2.5cm]{geometry}
\usepackage{setspace}
\theoremstyle{plain} 
\newtheorem{theorem}{Theorem}[section]
\newtheorem{lemma}[theorem]{Lemma}
\newtheorem{proposition}[theorem]{Proposition}
\newtheorem{corollary}[theorem]{Corollary}
\theoremstyle{definition}

\numberwithin{theorem}{section}

\newcommand{\R}{\mathbb{R}}
\newcommand{\C}{\mathbb{C}}
\newcommand{\Z}{\mathbb{Z}}

\renewcommand{\mod}{\operatorname{mod}}

\newcommand{\td}{\rightarrow}

\newcommand{\bb}[1]{\left(#1\right)}

\title{A density theorem for prime squares}

\author{Genheng Zhao}
\address{Academy of Mathematics and Systems Science, Chinese Academy of Sciences, Beijing 100190, China}
\email{zhaogenheng@amss.ac.cn}

\date{}

\begin{document}

\begin{abstract}
 Let $s\geq 8$ be an integer  and  $P$ be a set  of primes  with relative lower density greater than $\sqrt{1-\min\{s,16\}/32}$. We prove that every sufficiently large integer  $n\equiv s(\mod 24)$ can be represented by a sum of $s$ squares of primes in $P$.  
\end{abstract}

\maketitle

\section{Introduction and main results}
In 1938, Hua \cite{Hua} showed that for $s\geq 5$, every sufficiently large integer $n\equiv s (\mod  24)$ can be represented by  a sum of $s$ prime squares.  We prove here a density version of this, but for a smaller range $s\geq 8$. 

Let $\mathbb{\mathbb P}$ denote the set of primes and $P\subset \mathbb{P}$, we define the lower density of $P$ by
\begin{equation}
\label{eq: density-def1}
	\delta_P=\liminf_{N\td\infty} \frac{|\{ n\leq N: n\in P\}|}{|\{n \leq N: n\in \mathbb{P}\}|}.
\end{equation}  

\begin{theorem}
\label{thm:main}

Let $s\geq 8$ be an integer and $P\subset \mathbb{P}$ with $\delta_P> \sqrt{1-\min\{s,16\}/32}$. Then every sufficiently large integer $n\equiv s(\mod 24)$ can be represented by a sum of $s$ squares of primes in $P$.
\end{theorem}

The main tool we utilize here is the transference principle developed by Green  \cite{Green}, which allows us to consider our problem on pseudorandom sets. Roughly speaking, a set of integers is said to be pseudorandom  if its members are distributed randomly among congruences of small modules. According to this criterion, the set of odd primes is of course not pseudorandom. An obviously way to gain pseudorandomness is by the $W$-trick, that is, to consider sets of the from $\{n:Wn+b\in \mathbb{P}\}$ with $\prod_{p<w}p$ and $(W,b)=1$. Using this method, Li-Pan \cite{Pan} and  Shao \cite{Shao} succeeded in establishing density versions of  Vinogradov's theorem of three primes, with sharp density bounds. 

As for the problem of prime squares, pseudorandom sets should be of the form $\{n:Wn+b\in \mathbb{P}^2\}$.  Since $p^2\equiv 1(\mod 24)$ holds for all odd primes, we must have $24\mid W$. Thus the simplest module for prime squares to gain pseudorandomness  is likely to be $W= 8\prod_{2<p<w} p$.   This module has appeared in the work of Browning-Prendiville \cite{TD} for transferring general squares but not prime squares.
We will show that  it can  be used to transfer prime squares as well. Based on this, we follow a refined transference argument of Salmensuu \cite{Sal} to obtain a density theorem for  prime squares. 

The rest of this paper is organized  as follows. In section 2 we fix some notations. In Section 3 we use transference principle to prove      
Theorem \ref{thm:main}, assuming that we can transfer prime squares by $W=8\prod_{2<p<w}p$.  This will be verified  in the subsequent  two sections. 
\section{Notation}
 Let $\Z$ denote the set of integers, $\R$  the set of real numbers and $\C$ the set of complex numbers. Let $q$ be a positive integer, we will write $\Z_q=\Z/q\Z$. For $\alpha \in \R$, we use $\|\alpha\|$ to denote the distance between $\alpha$ and its nearest integer and $\lfloor \alpha \rfloor$ the largest integer less than $\alpha$.  Let $q\geq 1$ be an integer. We define $[q]=\{1,2,3,...,q\}$, $[q]^*=\{a\in [q]:(a,q)=1\}$ and $Z(q)$ the set of reduced quadratic residues $\{ b\in [q]^*:b\equiv h^2(\mod q)\}$.  Let $A$ be a set of integers and $f$ be a function on $A$, we denote $f(A)=\{f(a): a\in A\}$.  We will use $1_A$ to denote the characteristic function of $A$.
 
 	 For a finitely supported function $f:\Z\td \mathbb{\C}$, we define its Fourier transform by
 $$\hat f(\alpha)=\sum_{n}f(n)e(n\alpha),$$
 where $e(\cdot )=e^{2\pi i \cdot}$.  For a function $g:\Z/\R\td \mathbb{C}$ and $p>0$, we define its $L^p$ norm by 
 $$\|g\|_{L^p}=\begin{cases}
	\bb{\int_0^1|g(\alpha)|^pd\alpha }^{1/p},& 0<p<\infty, \\ 
	\sup_{\alpha \in \R/\Z}|g(\alpha)|,&p=\infty.
\end{cases}$$
and its Fourier transform by 
$$\hat g( k)=\int_{0}^1 g(\alpha )e(-k\alpha)d\alpha.$$
Let $h:\R/\Z\td \C$ be another function, we have the convolution 
$$g*h(\alpha)=\int_0^1 g(\alpha-u)\overline{h(u)}du. $$
  
   Let $f\in \C$ and $M,M_1,M_2\geq 0$. By $f\lesssim M$ we mean   $|f|\leq Cg$ holds for some constant $C>0$. The dependence of $C$ will appear in the subscript of $\lesssim $. The use of $\gtrsim $ is similar.  If we have both $M_1\lesssim  M_2$ and $M_1\gtrsim M_2$, we use $M_1\asymp M_2$ to denote it. Sometimes  we also use $f=O(M)$ to denote $f\lesssim M$ and by $f=o(M)$ we mean
 $$\lim_{M\td \infty} \frac{f}{M}=0.$$
  In particular, we use $f=o_M(1)$ to imply 
 $$ \lim _{M\td \infty} f=0.$$

 \section{ Transference  Argument}

 The first  ingredient is the following transference lemma,  which is a simple consequence of that of  \cite{Sal}.  The reader shall keep in mind that this is actually the final step of the whole argument. 
 \begin{lemma}[Transference Lemma]\label{lem:main}  Let $s\geq 2$, $M\geq 1$ be integers. For each $j\in [s]$, let $f_j$ and $\nu_j $ be functions on $[M]$, satisfying 
 \begin{enumerate}[(i)]
 	\item $0\leq f_j\leq \nu_j$,
 	 	\item  (Pseudorandomness) $\| \hat \nu_j-\widehat{1_{[M]}} \| _{L^{\infty}}= o\bb{ M}$,
 	\item ($q$-Restriction) $\|\hat f_j\|_{L^q}\lesssim_q  M^{1-1/q}$ holds for some $q\in (s-1,s)$.
 \end{enumerate}
 Assume that for some $\epsilon\in (0,1)$, they satisfy the mean condition
 $$\sum_{j\in [s]} \frac{\sum_{n\in [M]} f_j(n)}{M}\geq \frac{s}{2}+\epsilon   $$
 with
 $$\frac{\sum_{n\in [M]} f_j(n)}{M}\geq \epsilon$$
 for each $j\in [s]$.
 Let  $m$ be an integer  such that   \begin{equation}
  \label{eq:m-condition}
 	m=\frac{s}{2}M+o(M).
 \end{equation}
 Then for  $M\gtrsim_{\epsilon} 1$, there is a solution to 
 \begin{equation}
 \label{m-solution}
 	\begin{cases}m=\sum_{j\in [s]} n_j,\\
\prod_{j\in [s]}f_j(n_j)\neq 0.
 \end{cases}
 \end{equation}
 \end{lemma}
 
 \begin{proof}Let $M \gtrsim_{\epsilon }1 $  be such that all  conditions of \cite[Proposition 3.9]{Sal} are satisfied. As a result of this, we have
 \begin{equation*}
 	\sum_{m=n_1+n_2+\cdots n_s}\prod_{j\in [s]}f_j(n_j)\gtrsim_{\epsilon } N^{s-1}.
 \end{equation*}
 Since $m=\sum_{j\in [s]}n_j$ has $O(M^{s-1})$ solutions, by pigeonhole principle the above inequality ensures the existence of one solution for which 
 $$\prod_{j\in [s]}f_j(n_j) \gtrsim_{\epsilon }1>0 .$$ 
 \end{proof}

Let $N$ be a large integer and $w=O(\log\log \log N)$.   We  introduce the module $W=8\prod_{2<p<w} p$. For the rest of this paper, we let $N$, $w$ and $W$ be always as above. Only in the proof of Theorem \ref{thm:main} their values will be made explicit.

  Let $b\in Z(W)$ and $H(b)=\{h\in [W]^*:h^2\equiv b(\mod W)\}$. By Chinese remainder theorem it is clear that $|H(b)|=4\prod_{2<p<w} 2$. We denote this quantity by $H$, which satisfies $H|Z(W)|=\phi(W)$. To apply Lemma \ref{lem:main}, we  need to construct desired pairs of functions on the set $\{n\in [N]:Wn+b\in \mathbb{P}^2\}$. Let   $P\subset \mathbb{P}$ be as in the statement of Theorem \ref{thm:main}. We now define  
    \begin{equation}\label{eq:f-def}
 	f_b(n)=\begin{cases} \frac{\phi(W)}{WH}(2p\log p),&p^2=Wn+b\in P^2, n\in [N],\\ 0,&otherwise, \end{cases}
 \end{equation} 
and S
\begin{equation}\label{eq:nu-def}\nu_b(n)=\begin{cases} \frac{\phi(W)}{WH}(2p\log p),&p^2=Wn+b\in \mathbb{P}^2,n\in [N] ,\\ 0,&otherwise, \end{cases}\end{equation} 
where $\phi$ is Euler's totient function.  Obviously they satisfy $0\leq f_b\leq \nu_b$.
 The proof of  the  following two propositions are left to Section 4 and Section 5.   

\begin{proposition}\label{prop: pseudo}Let   $\nu_b$ be given by \eqref{eq:nu-def}.  We have for any $\epsilon >0$ that  $$\| \hat \nu_b-\widehat{1_{[N]}} \|_{L^{\infty}}=O_{\epsilon}(w^{-1/2+\epsilon}N),$$
\end{proposition}

\begin{proposition}\label{prop: restriction}
Let $f _b$ be given by \eqref{eq:f-def}.  We have for any  $q>4$ that 
$$\|\hat{ f_b}\|_{L^{q}}\lesssim_q N^{1-1/q}.$$
\end{proposition}

To verify the mean value condition, we need the a lemma on the additive structure of dense subsets of $Z(W)$, which is proved by the technique of downsets.  Our argument  is  motivated by \cite{Sal}. 

  Let $q$ be a square free number. For each $a\in \Z_q$ and $p\mid q$, we denote its $p$-coordinate $a^{(p)}\in \{0,1,2,..,p-1\}$ by $a^{(p)}\equiv a(\mod p)$.  We then define  
$$D(a)= \{b\in \Z_q: b^{(p)}\leq  a^{(p)}, \forall p\mid q\}.$$     
We say $A\subset \Z_q$ is a downset if $a\in  A$ implies  $D(a)\subset A$. It is clear that the addition of two downsets is still a downset. For convenience, we will temporarily  use   $sA$ to denote the $s$-fold addition $A+A+\cdots +A$. 
\begin{lemma}\label{lem:sum} Suppose $w\geq 4$ and  $E\subset Z(W)$. If $|E|>\frac{1}{2}|Z(W)|$, then for every $n\equiv 8(\mod 24)$, we have 
$$n\equiv \sum_{j=1}^8b_j (\mod W),$$
with  $b_j\in E$, $1\leq j\leq  8$.
\end{lemma}

 \begin{proof}Since $w\geq 4$ we can write $W=24W'$ with $ (W',24)=1$.  Note that for any $b\in Z(W)$ we have $b\equiv 1(\mod 24)$ and  for any $b\neq b'\in Z(W)$ we have $W'\nmid (b-b')$.  We can  thus regard $E$ as a subset of $\Z_{W'}$, and  by  Chinese remainder theorem it now suffices to show $8E=Z_{W'}$.  This is trivial if $W'=1$ and we  assume that $w\geq 6$. 
   
  Let $u\in \Z_{W'}$ be such that $u^{(p)}= (p-1)/2$, $\forall p\mid W'$. Then $|D(u-1)|=|Z(W')|$. By \cite[Lemma 5.8]{Sal}, there exists a downset $E'\subset D(u-1)$ such that $|E'|=|E|$ and $|8E|\geq |8E'|$. Noticing that  $(u-1)-E'\subset D(u-1)$   and $|E
 '|+|(u-1)-E'|>|Z(W)|=|Z(W')|=|D(u-1)|$, we conclude  $u-1\in E'+E'$. Since $E'+E'$ is still a downset, this means $D(u-1)\subset E'+E$. By the fact that  $5$ is the smallest prime factor of $W'$, it is easy to check $4D(u-1)=\Z_{W'}$ and thus $8E'=\Z_{W'}$.  Recall that $|8E|\geq |8E'|$, which soon  concludes the proof.
   \end{proof}

\begin{proof}[Proof of Theorem 1.1]

 Let $\lambda=\sqrt{1-\min\{s,16\}/32}$ and suppose that we have $\delta_P\geq \lambda +\kappa $ for some $\kappa\in (0,1/4)$.
For  $n\equiv s(\mod 24)$ we  set 
$w=\log \log \log n$     and  $N=\lfloor 2n/sW \rfloor$. It is clear that  as $N\td \infty$,
$$w=O(\log\log\log  N),\quad \frac{n}{W}=\frac{s}{2}N+o(N).$$
For each $b\in Z(W)$, we can now let $f_b$ be given by \eqref{eq:f-def}, and then consider its mean value \begin{equation}\label{eq: density-def2}
	\delta(N;b)=  \frac{\sum_{n\in [N]}f_b(n)}{N}.
\end{equation}  
 By prime number theorem in arithmetic progressions (see \cite[Theorem 8.16]{Ten}) and \eqref{eq: density-def1}, we have
$$\sum_{b\in Z(W)} \delta(N;b)=\frac{\phi(W)}{NWH}\sum_{p\in P\cap [\sqrt{W+b}, \sqrt{WN+b}]} 2p\log p\geq (1+o_N(1)) \frac{\phi(W)}{H} \delta_P^2  .$$
Recall that $H|Z(W)|=\phi(W)$.  The above  therefore yields
\begin{equation}\label{eq: density-eq}\begin{split}
	\frac{1}{|Z(W)|}\sum_{b\in Z(W)} \delta(N;b)\geq (1+o_N(1))\delta_P^2	\end{split}\geq \lambda^2+\frac{\kappa}{2}.
\end{equation}
 provided $N\gtrsim_{\kappa} 1$. Let $\mu=\delta(b_0;N)=\max_{b\in Z(W)} \delta(b;N)$. By pigeonhole principle, we have  for any $\epsilon \in (0,1/8)$ that $$\delta(N;b)\geq \frac{\lambda^2+\kappa/2-(1/2+\epsilon)\mu }{1/2-\epsilon}$$ 
 holds for at least a $(1/2+\epsilon)$ propotion of  values of $b\in Z(W)$.  Taking $\epsilon =\kappa/4$ this yields
 $$ \left|\left\{b\in Z(W): \delta(N;b)\geq 2\lambda^2-\mu+\frac{\kappa}{4}\right\}\right|\geq \bb{\frac{1}{2}+\frac{\kappa}{4}}|Z(W)|.$$ 
 Now we can apply  Lemma \ref{lem:sum}  and obtain    $n-(s-8)b_0\equiv \sum_{j=1}^8 b_j(\mod W)$ with $\delta(N;b_j)\geq 2\lambda^2-\mu+\kappa/4$, $1\leq j \leq 8$.  Note that  $$ (s-8)\delta(b_0;N)+\sum_{j=1}^8 \delta(b_j;N)\geq  (s-16)\mu+16\lambda^2+2\kappa$$
 and 
 $$ \lambda^2+\frac{\kappa}{2}\leq \mu\leq 1.$$
A case-by-case argument on $s$ will give
 $$  (s-8)\delta(b_0;N)+\sum_{j=1}^8 \delta(b_j; N)\geq \frac{s}{2}+\kappa.$$

    We now expand the set $\{b_1,...,b_8\}$ to a set of  $s$ elements  by adding $(s-8)$ copies of $b_0$, and then relabel its members by $\{b_j\}_{j\in [s]}$.  The value of 
  \begin{equation}\label{eq: decomp-m}
 	m=\frac{n-\sum_{j\in [s]} b_j}{W}
 \end{equation} 
  is obviously $sN/2+o(N)$.
  Let $f_j$ by $\nu_j$ be defined  by  \eqref{eq:f-def} and \eqref{eq:nu-def}   respectively, with $b=b_j$, $j\in [s]$.  The above argument shows that they satisfy the density condition of Lemma \ref{lem:main}, with $\epsilon$ replaced by $\kappa$ and $M$ replaced by $N$.  By Proposition \ref{prop: pseudo} and  Proposition \ref{prop: restriction}, we have the desired pseudorandomness and $q$-restriction as $s-1\geq 7>4$.  Now we apply  Lemma \ref{lem:main}  to this $m$  and then obtain  a solution $\{n_j\}_{j\in [s]}$ to \eqref{m-solution}. 
 By definition, $f_{j}(n_j)> 0$ implies $Wn_j+b_j\in P^2$, $j\in [s]$. Hence by \eqref{eq: decomp-m} we conclude the proof.

\end{proof}

\section{Pseudorandom condition}
 Let   $\nu=\nu_b$  be given by \eqref{eq:nu-def} with  $b\in Z(W)$  fixed.  Let  $A\geq 1$ and $Q=(\log N)^A$. We now identify $\R/\Z$ with the interval $(Q/N,1+Q/N]$ and then define the major arcs by $$\mathfrak{M}=\bigcup_{\substack{1\leq q\leq Q\\ a\in [q]^*}}\mathfrak{M}(q,a)=\{\alpha \in \R/\Z: |\alpha-a/q|\leq Q/qN\}.$$ 
  and define the minor arcs by $\mathfrak{m}=(Q/N,1+Q/N]\setminus \mathfrak{M}$. We will always let  $N\gtrsim 1$ to ensure $N>2Q^2$.

Since $p^2\equiv b (\mod W)$ if and only if $p\equiv h(\mod W)$ for some $h\in H(b)$, $\hat \nu(\alpha)$ has the decomposition
\begin{equation}\label{eq: decomp-nu}
	\begin{split}
	\hat \nu(\alpha)&=\sum_{n\in [N]}\nu(n)e(n\alpha)
	\\&=	\frac{\phi(W)}{WH}\sum_{\substack{p\in  \sqrt{W[N]+b}} }(2p\log p)e\bb{\frac{p^2-b}{W}\alpha }
	\\&=\frac{\phi(W)}{WH}\sum_{h\in H(b)}  \sum_{\substack{p\in  \sqrt{W[N]+b} \\ p\equiv h (\mod W) }}(2p\log p) e\bb{\frac{p^2-b}{W}\alpha}
	\\ &=\frac{\phi(W)}{WH} \sum_{h\in H(b)} \sum_{\substack{p\leq \sqrt{WN+b}\\ p=Wm+h}}(2p\log p)e\bb{\frac{(Wm+h)^2-b}{W}\alpha}+O(W).
	\end{split}
\end{equation}
Write  $\alpha =a/q+\beta$ with $q\leq Q$ and $a\in [q]^*$.  Then  by Siegel-Walfisz theorem (see \cite[Theorem 8.17]{Ten} ) and partial integration we have   \begin{equation*}
\begin{split}
&\sum_{\substack{p\leq \sqrt{WN+b}\\ p=Wm+h}}(2p\log p)e\bb{\frac{(Wm+h)^2-b}{W}\bb{\frac{a}{q}+\beta}}\\=&\sum_{l\in  [q]} e\bb{\frac{((Wl+h)^2-b)a}{qW}}e\bb{-\frac{b \beta}{W}} \sum_{\substack{p\leq \sqrt{WN+b}\\ p\equiv Wl+h(\mod qW)}}(2p\log p)e\bb{\frac{p^2}{W}\beta}\\ =&\sum_{\substack{l\in [ q]\\ (Wl+h,qW)=1}} e\bb{\frac{((Wl+h)^2-b)a}{qW}}(1+O(|\beta|)) \frac{W}{\phi(qW)} \int_0^Ne(\beta u)du 
\\&+O((1+|\beta|N)Ne^{-c_A\sqrt{\log N}})	
\end{split}
\end{equation*}
where $C_A$  an ineffective constant depending only on $A$.   
Inserting this back to \eqref{eq: decomp-nu} and  expanding the square, we arrive at the following  asymptotic formula on major arcs.
 \begin{proposition}\label{prop: type-1-aym} Let  $\alpha\in \mathfrak{M}(q,a)$ with $q\leq Q$ and $a\in [q]^*$. We have
 $$\hat \nu(\alpha)=\frac{\phi(W)}{\phi(qW)}S(q,a)\int_0^Ne \bb{ \bb{\alpha-a/q}u}du+O_A\bb{N(\log N)^{-2A}},$$
 where
 \begin{equation}\label{eq: def-S}
 	S(q,a)=\frac{1}{H}\sum_{h\in H(b)} e\bb{\frac{h^2-b}{qW} }\sum_{\substack{l\in [q]\\(Wl+h,qW)=1}}e\bb{\frac{(Wl^2+2hl)a}{q}}.
 \end{equation}
 \end{proposition}
To proceed further we need to analyze the value of  $S(q,a)$. It  will  be shown  that  $S(q,a)$ is in fact a sum of Gauss sums of the form
 $$G(k,r)= \sum_{\substack{l^*\in [ k]^*}}e\bb{ \frac{r {l^*}^2}{k}},$$
 which obey the bound (see \cite[Theorem 8.5]{Hua2})
 
 \begin{equation}\label{eq: Gauss}
 	|G(k,r)|\lesssim_{\epsilon} k^{1/2+\epsilon}
 \end{equation}
 for any $\epsilon >0$ and $(k,r)=1$.
 
\begin{lemma}\label{lem: exp} Let $S(q,a)$ be as in \eqref{eq: def-S}.
\begin{enumerate} [(i)]
\item If $(q,W)\nmid 2$, then $S(q,a)=0$.
\item If $(q,W)=1$ and let $W^{-1}$ be the inverse of $W(\mod q)$, then $$S(q,a)=e\bb{-\frac{W^{-1}ba}{q}}G(q,W^{-1} a).$$	
\item If $(q,W)=2$ and let $(2W)^{-1}$ be the inverse of $2W(\mod q/2)$, then
$$ S(q,a)=\frac{2}{H} \sum_{h\in H(b)}e\bb{\frac{(h^2-b)a}{qW}-\frac{(2W)^{-1}h^2a}{q/2}} G(q/2,(2W)^{-1}a).$$
\item If $q=2,$ then $S(q,a)=0$.
\end{enumerate}	
\end{lemma}

\begin{proof} 

\begin{enumerate}[(i)]
\item Since $(Wl+h,W)=(h,W)=1$, the inner sum can be written as
 \begin{equation}\label{eq: exp-decomp}
 	 \sum_{\substack{l\in  [q]\\ (Wl+h,q)=1}} e\bb{\frac{(Wl^2+2hl)a}q}=\sum_{l^*\in [ q]^*}  \sum_{Wl+h\equiv l^* (\mod q)}e\bb{\frac{(Wl^2+2hl)a}q}
 \end{equation}
 Let  $t=(q,W)$, when $t\mid (l^*-h)$, $Wl+h\equiv l^* \mod q $ is solvable with solutions of the form
 \begin{equation}\label{eq: roots}
 	 l= yq/t+z, \quad y\in [t],
 \end{equation}
 from which we derive that 
 \begin{equation}\label{eq: canc}
 	\begin{split}
 	   \sum_{Wl+h\equiv l^* (\mod q)}e\bb{\frac{(Wl^2+2hl)a}q}&=\sum_{y\in [ t]}e\bb{\frac{(Wz^2+2hq/t+2hz)a}{q}}
 	   \\&=e\bb{\frac{(Wz^2+2hz)a}{q}} \sum_{y\in  [t]}e\bb{\frac{2ha}{t}},
 	\end{split}
 \end{equation}
 This vanishes if $t\nmid 2$.
 \item When $t=1$, let $W^{-1}$ be the inverse of $W(\mod q)$.  Then  $z=W^{-1}(l^*-h)$ in \eqref{eq: roots}. We have
  $$ Wz^2+2hz\equiv W^{-1}\bb{{l^*}^2-h^2} (\mod q)$$
  and thus by the fact $(h^2-b)/W \equiv W^{-1}(h^2-b)(\mod q)$ and \eqref{eq: def-S}-\eqref{eq: canc}, 
  $$S(q,a)=e\bb{\frac{-W^{-1}ba}{q}} G(q,W^{-1}a).$$
 \item When $t=2$, since $2\mid (l^*-h)$, $Wl+h\equiv l^* \mod q$ is always sovable. Let $W_0=W/2$, $q_0=q/2$  and $W_0^{-1}$ be the inverse  of  $W_0(\mod q_0)$. Since $4\mid W$, we know that $q_0$ is odd, which means  $2^{-1}(\mod q_0)$ exist .   We have now $z=(2W_0)^{-1}(l^*-h)$ and 
$$ e\bb{\frac{(Wz^2+2hz)a}{q} }=e\bb{\frac{(W_0z^2+hz)a}{q_0}}=e\bb{\frac{(4W_0)^{-1}({l^*}^2-h^2)a}{q_0}}.$$
 We collect these facts and derive from \eqref{eq: def-S}-\eqref{eq: canc} that
$$ S(q,a)=\frac{2}{H} \sum_{h\in H(b)}e\bb{\frac{(h^2-b)a}{qW}-\frac{(2W)^{-1}h^2a}{q_0}} \sum_{\substack {l^*\in [ q]^*}}e\bb{\frac{(2W)^{-1}a{l^*}^2}{q_0}}.$$
Given $l_1^*, l_2^* \in [ q]^*$ with $(l_1^*l_2^*,q)=1$ and $l_1^*\neq l_2^*$. Since $2\mid (l_1^*-l_2^*)$, we can not have $q_0\mid (l_1^*-l_2^*)$. Thus $l^*$ travels through $\phi(q)$ different values $(\mod q_0)$.  By the fact   $\phi(q)=\phi(q_0)$ ,  we have $$\sum_{\substack {l^*\in [q]^*}}e\bb{\frac{(2W)^{-1}a{l^*}^2}{q_0}}=G(q_0,(2W)^{-1} a).$$
\item When $q=2$ we have $a=1$ and only need to consider 
 $$S(2,1)=\frac{2}{H}\sum_{h\in H(b)} e\bb{\frac{h^2-b}{2W}}.$$
  Write $W=8K$ with $K=\prod_{2<p<w}p$. Then each $x\in [W]$ is of the form $u+4Kv$ with $v\in \{0,1\}$ and $u\in [4K]$. Moreover, since
   $$x^2\equiv u^2+8K(uv+2Kv)\equiv u^2 (\mod W),$$
     we obtain
   $$ S(2,1)=\frac{1}{H} \sum_{\substack{u\in [4K]\\u^2\equiv b (\mod W)\\ }} e\bb{\frac{u^2}{16K}}\sum_{v\in \{0,1\}} e\bb{\frac{uv}{2}}.$$
   Note that $u^2\equiv b ( \mod W)$ implies  $u$ is odd. The above sum thus vanishes.  
\end{enumerate}
\end{proof}
 Combing Lemma \ref{lem: exp} and \eqref{eq: Gauss}, the following estimate of $S(q,a)$ is now available.

\begin{corollary}\label{cor: exp}For any $\epsilon >0$ and $(a,q)=1$,  we have  $|S(q,a)|\lesssim_{\epsilon}  q^{\frac{1}{2}+\epsilon}$. In particular, when  $(q,W)\nmid 2$ or $q=2$, we have  $S(q,a)=0$. \end{corollary}

\begin{proposition}[Major-arc estimate]\label{prop: upper-1}  Let $\alpha \in \mathfrak{M}(q,a)$ with $q\leq Q$ and $a\in [q]^*$. For any $\epsilon>0$ we have 
$$\hat \nu (\alpha)=\begin{cases}\int _0^N e((\alpha-1) u)du+O_A(N(\log N)^{-2A}),&q=1,\\  O_{A,\epsilon}(Nw^{-{1/2+\epsilon }}), & q\neq 1,
	
\end{cases}
$$ 
and 
$$|\hat \nu (\alpha)|\lesssim_{A,\epsilon } Nq^{-{1/2}+\epsilon}(1+N|\alpha -a/q|)^{-1} .$$ 

\end{proposition}

\begin{proof}By  Proposition \ref{prop: type-1-aym}, we have
$$\hat \nu(\alpha)=\frac{\phi(W)}{\phi(qW)}S(q,a)\int_0^Ne \bb{ \bb{\alpha-a/q}u}du+O_A\bb{N(\log N)^{-2A}}.$$
By Corollary \ref{cor: exp}, we have for any $\epsilon >0$ that
$$\frac{\phi(W)}{\phi(qW)}S(q,a) =\begin{cases}
 	1,&q=1,\\0, &q\in (1,w),
 	\\ O_{\epsilon}(q^{-1/2+\epsilon}), &otherwise. 
 \end{cases}
$$
From this we derive the desired asymptotic formula for $q=1$.  While  $q\neq 1$, this leads to
$$|\hat \nu(\alpha)|\lesssim_{\epsilon} w^{-1/2+\epsilon}N+N(\log N)^{-2A}\lesssim_{\epsilon,A} w^{-1/2+\epsilon} N ,$$
since $w=O(\log \log \log N)$.  As for the second claim, notice that 
\begin{equation*}
\begin{split}
	\int _0^N e((\alpha-a/q) u)du&=\sum_{n=1}^Ne(n(\alpha-a/q))+O(|\alpha-a/q|N)
	\\&\lesssim \min\{N,|\alpha-a/q|^{-1}\}+O((\log N)^A)
	\\ &\lesssim N(1+N|\alpha-a/q|)^{-1},
\end{split}
\end{equation*}
and 
$$ N(\log N)^{-2A}\leq  N(\log N)^{-(1/2+\epsilon)A -A}\lesssim Nq^{-1/2+\epsilon}(1+N|\alpha-q/a|)^{-1}.$$

\end{proof}

 The estimate on minor arcs is quite classical and we can start from the following result of Hua.

 \begin{lemma}\label{lem: hua-expo} Suppose $0< W\leq (\log X)^{\sigma_1}$ and $(\log X)^{\sigma}\leq q \leq X^2(\log X)^{-\sigma}$. Given any  $\sigma_0\geq 1$ and $(h,q)=1$, we have
 $$ \sum_{\substack{p\leq X\\ p\equiv t(\mod W) }} e\bb{\frac{hp^2}{q}}\lesssim \frac{X}{W(\log X )^{\sigma_0}}$$
 holds for  $\sigma\gtrsim \sigma_0+\sigma_1$.
 \end{lemma}
 \begin{proof}
 See 	\cite[Theorem 10]{Hua2}.
 \end{proof}

 \begin{proposition}[Minor-arc estimate]  \label{prop: Type-2} Let $\alpha \in \mathfrak{m}$. Then  $$\hat\nu(\alpha )\lesssim \frac{N}{(\log N)^B}$$
 holds for $A\gtrsim B$.
   
 \end{proposition}
\begin{proof} By \eqref{eq: decomp-nu},  it suffices to show this for 
$$\sum_{\substack{p\leq \sqrt{WN+b}\\ p\equiv h(\mod W)}} (2p\log p)e\bb{\frac{p^2\alpha}{W}}.$$By Dirichlet approximation theorem, there exists $q\leq N/Q$ and $(a,q)=1$ such that $|\alpha -a/q|\leq Q/qN$. Recall that we have $N> 2Q^2$, thus if $q\leq Q$, $a$ must belong to $[q]^*$ and hence $\alpha \in \mathfrak{M}$.  Since $\alpha \not \in \mathfrak{M}$, this means  $q\in [Q,N/Q]=[(\log N)^{A},N(\log N)^{-A}]$ and $|\alpha-a/q|\leq 1/N$.

  	Notice that when    $X\in [\sqrt{N}(\log N)^{-A/4}, \sqrt{N}\log N]$ and $N\gtrsim _A 1$, we have  $$q\in[(\log N)^A,N(\log N)^{-A}]\subset [(\log X)^{A/4}, X^2(\log X)^{-A/4}].$$
Since $W=(\log N)^{o_N(1)}$, by Lemma \ref{lem: hua-expo},  we let $a_1=a/(a,W)$ and $W_1=W/(a,W)$ to obtain for $A\gtrsim B$ and $X\in [\sqrt{N}(\log N)^{-A/4}, \sqrt{N}\log N]$ that 
$$ \sum_{\substack{p\leq X\\ p\equiv h(\mod W)}} e\bb{\frac{ap^2}{qW}}= \sum_{\substack{p\leq P\\ p\equiv h(\mod W)}} e\bb{\frac{a_1p^2}{qW_1}}\lesssim \frac{X}{ (\log N)^{10B}}.$$
Integrating by parts and using the  fact  $|\alpha-a/q|\leq 1/N$,  this gives for $A\gtrsim B $ and $X\leq \sqrt{N}\log N$ that 
$$ \sum_{\substack{p\leq P\\ p\equiv h(\mod W)}} (2p\log p)e\bb{\frac{p^2\alpha}{W}}\lesssim \frac{X^2}{(\log N)^{5B}}+\frac{N}{(\log N)^{A/8}}.$$
Again by $ W=(\log N)^{o_N(1)}$,  we see $\sqrt{WN+b}\leq  \sqrt{N} \log N$ holds for $N\gtrsim 1$. We can thus conclude the proof by choosing $X=\sqrt{WN+b}$ and $A\gtrsim B$.
\end{proof}

\begin{proof}[Proof of Proposition 3.2]
We use a similar method to estimate $\widehat {1_{[N]}}(\alpha)$. Let  $\alpha \in \mathfrak{M}(q,a)$ with $q\leq Q$ and $a\in [q]^*$. We have
$$
\widehat{1_{[N]}}(\alpha)= 1_{q=1} \int_{0}^N e((\alpha-1)u)du +O((\log N)^A)
 $$
 and for $\alpha \in \mathfrak{m}$  that
$$\widehat{1_{[N]}}(\alpha)=O\bb{N(\log N)^{-A}}.$$
Comparing  these two facts with Proposition \ref{prop: upper-1}  and Proposition \ref{prop: Type-2}, we conclude by $\epsilon>0$, $B=1$ and $A\gtrsim 1$ that 
$$\|\hat \nu -\widehat{1_{[N]}} \|_{L^\infty}\lesssim_{\epsilon}  w^{-1/2+\epsilon}N$$

\end{proof}

\section{Restriction estimate}
In this section we again let $\nu=\nu_b$ be as in \eqref{eq:nu-def} with  $b\in Z(W)$ fixed. Due to technical reasons we will not fix $f_b$ but instead adopt a flexible argument for  all functions bounded by $\nu$.  From the  work of Bourgain \cite{Bourgain}, we notice that   it will be convenient to argue on discrete sets.

For any function $f:[N]\td \C$, we define for $u>0$ that  \begin{equation}\label{eq :level}
	f_*(u) =|\{n\in [N]:|\hat f(n/N)|\geq u N\} |.
\end{equation}
 The following lemma, which has essentially appeared in \cite{Bourgain},  can be used to control the the size of $f_*(u)$, provided $u$ is not too small. 

  \begin{lemma} \label{lem: Bourgain} Let $ Q, T, M\geq 2$ and    $g$ be a function on $\R/\Z$  that obeys the bound
\begin{equation}\label{eq: condition}
	|g(\alpha)|\leq \max\left\{ \sum_{q\leq Q}\frac{1}{q} \sum_{a\in [q]}  \bb{M\|\alpha-a/q\|+1}^{-2}, \frac{1}{T} \right\}.
\end{equation}
Let $\{t_r\}_{r\in [R]}\subset \R/\Z $ be  such that $\|t_r-t_s\|\geq M^{-1}$  holds for $r\neq s$ and
\begin{equation}\label{eq: large-value}
	\sum_{r,s\in [R]} |g(t_r-t_s)| \geq \eta R^2
\end{equation}
holds for some $\eta\geq 2/T$.  For any $\epsilon>0$ we have   $$R\lesssim_{\epsilon} \eta^{-1}(\log Q)^{O_{\epsilon}(1)}T^{\epsilon}.$$
\end{lemma}
\begin{proof} Let $$F(\alpha)=(M\|\alpha\|+1)^{-2},\quad G(\alpha)=\sum_{q\leq Q}\frac{1}{q} \sum_{a\in [q]}F(\alpha-a/q).$$
 By \eqref{eq: condition}, \eqref{eq: large-value} and  $\eta\geq 2/T$, we have
 \begin{equation}
 	\eta R^2\leq 2 \sum_{\substack{r,s\in [R]\\ |g(t_r-t_s)|\geq 1/T}} |g(t_r-t_s)|\leq 2\sum_{r,s\in [R]} G(t_r-t_s).
 \end{equation}
 Let $$ \psi(\alpha)=\bb{\frac{1}{2M+1}\sum_{|k|\leq M} e(k\alpha)}^2,\quad \sigma(\alpha)=\sum_{r\in [R]} \psi(\alpha-t_r).$$
 Notice that $\psi(\alpha)\lesssim \min\{1,(M\|\alpha\|)^{-2}\}$. It follows immediately $\|\psi\|_{L^1}\lesssim M^{-1}$ and thus $\|\sigma \|_{L^1} \lesssim RM^{-1}$. In particular, since $\|t_r-t_s\| \geq M^{-1}$ for $r\neq s$, we have $\|\sigma\|_{L^{\infty}}\lesssim 1$ and hence $\|\sigma\|_{L^{2}}^2\lesssim RM^{-1}$,  which provides the crucial saving. 

  Meanwhile, notice also that  $\psi(\alpha)\gtrsim  1$ for $\|\alpha\| \leq (10 0M)^{-1}$. We then have $\sigma(\alpha)\gtrsim 1$ if $\|\alpha-t_r\|\leq (100 M)^{-1}$ for some $r\in [R]$. Denote $\tilde {\sigma}(\alpha)=\sigma(-\alpha)$. Then  for $\|\alpha-(t_r-t_s)\|\leq (1000M)^{-1}$, we have
 $$\tilde \sigma*\sigma(\alpha)\gtrsim M^{-1}.$$  
 When $\|\alpha-\alpha'\|\leq (1000M)^{-1}$, it is clear $F(\alpha) \asymp F(\alpha')$  and hence $G(\alpha)\asymp G(\alpha')$. Thus 
 $$\eta R^2 M^{-2}\lesssim   \int_0^1 G(\alpha )(\tilde \sigma*\sigma)(\alpha)d\alpha=\sum_{|k|\leq 2N}\hat G(k) |\hat \sigma(k)|^2.$$
By the simple fact  $|\hat F(k)|\leq \| F\|_{L^1}\lesssim M^{-1}$, we have
 $$\hat G(k)=\sum_{q\leq Q}\frac{1}{q}\sum_{a\in [q]}e\bb{\frac{ak}{q}} \hat F(k)\leq \sum_{\substack{q\leq Q\\ q\mid k}}\hat F(k)\lesssim M^{-1}\sum_{\substack{q\leq Q\\ q\mid k}} 1 ,$$
 which means that 
 $$\eta R^2M^{-1}\lesssim  \sum_{|k|\leq 2M}  |\hat \sigma(k)|^2\sum_{\substack{q\leq Q\\ q\mid k}}1. $$
 We divide this sum into two parts, according to the size of inner sum. When the inner sum is smaller than $Q_0$, we use Plancherel's formula to obtain 
 $$ \sum_{|k|\leq 2N}  |\hat \sigma(k)|^2\sum_{\substack{q\leq Q\\ q\mid k}}1\leq Q_0\|\sigma\|_{L^2}^2\leq Q_0RM^{-1}. $$
 Otherwise, we use the fact $|\sigma(k)|\leq \|\sigma\|_{L^1}\lesssim RM^{-1}$ to obtain for any $r\geq 1$ that
 $$ \sum_{|k|\leq 2M}  |\hat \sigma(k)|^2\sum_{\substack{ q\leq Q \\ q\mid k}}1\lesssim  R^2M^{-2}Q_0^{-(r-1)} \sum_{k\leq 2M} \bb{ \sum_{ \substack{ q\leq Q\\ q\mid k}}1}^r\lesssim R^2M^{-1}Q_0^{-(r-1)}(r\log Q)^{2^r}, $$
 since
 \begin{equation*}
 	\begin{split}
 	 \sum_{k\leq 2M} \bb{ \sum_{ \substack{ q\leq Q\\ q\mid k}}1}^r &= 2M \sum_{q_1,q_2,..,q_r\leq Q} \frac{1}{[q_1,q_2\cdots ,q_r]}+O(Q^r)	\\ &\lesssim M \sum_{q\leq Q^r} \frac{(\sum_{d\mid q }1)^r}{q}+Q^r
 	 \\ &\lesssim  	M(r\log Q)^{2^r}.
 	\end{split}
 \end{equation*}
 We now arrive at
 $$ \eta R\lesssim  Q_0+ Q_0^{-(r-1)}(r\log Q)^{2^r}R. $$
 It then concludes the proof by  $Q_0=r(r\log Q)^{2^r/(r-1)}T^{1/(r-1)}$  and $r\gtrsim \epsilon^{-1} $.
\end{proof}
For those small values of $u$, our argument will  be based on the following simple fourth-moment estimate. 
\begin{lemma}\label{lem: critical value} Let $f$ be a function bounded by $\nu$. We have for $N\gtrsim 1$ that
$$\sum_{n\in [N]}|\hat f(n/N) |^4 \lesssim N^4 (\log N)^8.$$
\end{lemma}
\begin{proof}
Note  first that 
\begin{equation*}
\begin{split}
	\sum_{n\in [N]}|\hat f(n/N)|^4&=\sum_{n\in [N]}\sum_{m_1,m_2,m_3,m_4\in [N]}{f(m_1)}\overline {f(m_2)} \overline {f(m_3)} f(m_4) e\bb{\frac{(m_1-m_2-m_3+m_4)n}{N}}\\& = N\sum_{k}\left |\sum_{m-n=k} f(m)\overline {f(n)} \right |^2\\ & \leq N \bb{\frac{\phi(W)}{W}\sqrt{WN+b}\log (WN+b)}^4  \sum_{k} \left |\sum_{\substack{m-n=k\\ \nu(m)\nu(n)\neq 0}  }1\right|^2.\end{split}
\end{equation*}
The condition $\nu(m)\nu(n)\neq 0$ means that $ Wm+b$, $Wn+b\in \mathbb{P}^2$ and $m,n\in [N]$. Thus, for $k=0$, the inner sum contributes $O(N)$, and for $1\leq |k|< N$ we instead have
$$ \sum_{\substack{m-n=k\\ \nu(m)\nu(n)\neq 0}} 1\leq  \sum_{\substack{x^2-y^2=Wk\\ W\mid (x-y)  }} 1\leq \sum_{d\mid |k|} 1.$$
This leads to 
$$  \sum_{k} \left |\sum_{\substack{m-n=k\\ \nu(m)\nu(n)\neq 0}  }1\right|^2 \lesssim N+\sum_{1\leq k< N} \bb{\sum_{d\mid k}1}^2\lesssim N (\log N)^{3}.$$
By the fact $W=(\log N)^{o_N(1)}$, we obtain for $N\gtrsim 1$,
$$\sum_{n\in [N]}|\hat f(n/N)|^4\lesssim N^4(\log N)^{8}$$
as desired.

\end{proof}

\begin{proof}[Proof of Proposition 3.3]  Let $q>4$ and $f:[N]\td \C$ be bounded by $\nu$.  Let $f_*(u)$ be given by \eqref{eq :level}. We first find by Lemma \ref{lem: critical value}, 
$$ (uN)^{4}f_*(u)\lesssim \sum_{n\in [N]} |\hat f(n/N)|^4\leq  N^4(\log N)^{8},$$
 It then follows 
	\begin{equation}\label{eq: lower-frequency}
	f_*(u)\lesssim u^{-4}(\log N)^{8} 
	\end{equation}
On the other hand, we let $\{t_r\}_{r\in [R]}=\{n/N: |f(n/N)|\geq uN,  n\in [N]\}$  and  $\{c_r\}_{r\in [R]}\subset \C$ be such that $c_r\hat f(t_r)=|\hat f(t_r)|$, $r\in [R]$. Then,
$$ uRN\leq \sum_{r\in [R]}|\hat f(t_r)|=\sum_{r\in [R]}c_r\hat f(t_r)=\sum_{n\in [N]}f(n) \sum_{r} c_re(nt_r).$$
By Cauchy-Schwarz inequality and  $|f|\leq \nu$, this yields
$$u^2R^2N^2\leq \bb{\sum_{n\in N}|f(n)|}\bb{\sum_{n\in [N]} |f(n)|\sum_{r,s\in [R]}e(n(t_r-t_s))}\lesssim N \sum_{r,s\in [R]} \hat \nu(t_r-t_s).$$
Next, by H\"older's inequality, we have for any $q'>4$ that 
$$ u^{q'} R^2\leq \sum_{r,s} \bb{N^{-1}\hat \nu(t_r-t_s)}^{q'/2}.$$
Observe that by Proposition \ref{prop: type-1-aym} and  Proposition \ref{prop: Type-2},  for $A\gtrsim B\geq 1$, there is some constant $C_{q',A,B}$ such that  $ C_{q',A,B}(N^{-1}\hat \nu)^{q'/2} $ satisfy	\eqref{eq: condition} with $Q=(\log N)^A$, $T=(\log N)^B$ and $M=N$. By Lemma \ref{lem: Bourgain}, we have for any $\epsilon>0$ and $u\geq 2(\log N)^{-B}$ that 
$$ R\lesssim_{\epsilon,{q'},A,B}  u^{-q'} (\log N)^{(B+1)\epsilon},  $$
which leads to the  bound
\begin{equation}\label{eq: higher-frequency}
	 f_*(u) \lesssim_{\epsilon,{q'},A,B}  u^{-q'} (\log N)^{(B+1)\epsilon}\lesssim_{\epsilon,q',A,B} u^{-q'-2\epsilon}.
\end{equation}
When $u\leq 2(\log N)^{-B}$, we use \eqref{eq: lower-frequency} to obtain 
$$  f_*(u)\lesssim_{A,B}u^{-4-8B^{-1}}.$$
Since  $q'+2\epsilon$ and $4+8B^{-1}$ can be arbitrarily close to $4$, we conclude that   $f_*(u)\lesssim_q u^{-(q+4)/2} $  holds for $u>0$.
Note that  $\|\hat f\|_{L^{\infty}}< 2N$ as $N\gtrsim 1$. We have now 

\begin{equation*}
	\begin{split}
		\sum_{n\in [N]}|\hat f(n/N)|^q &\lesssim (2N)^q \sum_{k\in \Z}2^{kq}|\{n\in [N]: 2^{k}N\leq |\hat f(n/N)|\leq 2^{k+1}\}|
		\\&\lesssim_q N^{q}\sum_{k\leq 0} 2^{kq}f_*(2^k)\\&\lesssim_q N^q \sum_{k\leq 0} 2^{k(q-4)/2}\\&\lesssim _q N^q.
			\end{split}
\end{equation*}
For any $\theta\in [0,N^{-1}]$, we now specialize $f(n)=f_b(n)e(n\theta)$ and thus 
$$\hat f (\alpha)= \sum_{n\in [N]}f_b(n)e(n(\alpha+\theta))=\hat f_b(\alpha+\theta).$$
Since $|f|\leq \nu$, the above argument gives
$$\sum_{n\in [N]} |\hat f_b(n/N+\theta)|^4\lesssim_q N^q.$$
Integrating this over $\theta \in [0,N^{-1}]$ we conclude the proof.

\end{proof}


\begin{thebibliography}{9}
\bibitem{Green}
B. Green, \textit{Roth's theorem in the primes}, Ann. of Math. (2) \textbf{161}(3) (2005), 1609--1636. 

\bibitem{Ten} G. Tenenbaum, \textit{Introduction to analytic and probabilistic number theory}, Graduate Studies in Mathematics, American Mathematical Society, Providence, RI, 2015.


\bibitem{Pan}H. Li and H. Pan, \textit{A density version of Vinogradov’s three primes theorem}, Forum Math. \textbf{22} (2010), 699--714.


\bibitem{Bourgain}J. Bourgain, \textit{On $\Lambda(p)$-subsets of squares}, Israel J. Math. \textbf{67} (1989), 291--311.

\bibitem{Sal}J. Salmensuu, \textit{A density version of Waring’s problem}, Acta Arith. \textbf{199} (2021), 383--412.

\bibitem{Hua}L. K. Hua, \textit{Some results in the additive prime number theory}, Quart. J. Math. \textbf{9} (1938), 68--80.

\bibitem{Hua2}L. K. Hua, \textit{Additive Theory of Prime Numbers}, Translations of Mathematical Monographs, vol. 13, American Mathematical Society, Providence, RI, 1965.


\bibitem{TD}
T. D. Browning and S. M. Prendiville, \textit{A transference approach to a Roth-type theorem in the squares}, International Mathematics Research Notices \textbf{2016} (2016), 1--30.



\bibitem{Shao}
X. Shao, \textit{A density version of the Vinogradov three primes theorem}, Duke Math. J. \textbf{163}(3) (2014), 489--512.



\end{thebibliography}
\end{document}